\documentclass[12pt,reqno]{article}

\usepackage[usenames]{color}
\usepackage{amssymb}
\usepackage{amsmath}
\usepackage{amsthm}
\usepackage{amsfonts}
\usepackage{amscd}
\usepackage{graphicx}
\usepackage{algorithm2e}
\usepackage[colorlinks=true,
linkcolor=webgreen,
filecolor=webbrown,
citecolor=webgreen]{hyperref}
\usepackage{color}
\usepackage{fullpage}
\usepackage{float}
\usepackage{amsthm}
\usepackage{graphics}
\usepackage{latexsym}
\usepackage{epsf}
\definecolor{webgreen}{rgb}{0,.5,0}
\definecolor{webbrown}{rgb}{.6,0,0}
\newcommand{\seqnum}[1]{\href{https://oeis.org/#1}{\underline{#1}}}

\begin{document}

\begin{center}
\epsfxsize=4in
\end{center}

\theoremstyle{plain} \newtheorem{theorem}{Theorem}
\newtheorem{corollary}[theorem]{Corollary} \newtheorem{lemma}[theorem]{Lemma}
\newtheorem{proposition}[theorem]{Proposition}
\theoremstyle{definition} \newtheorem{definition}[theorem]{Definition}
\newtheorem{example}[theorem]{Example}
\newtheorem{conjecture}[theorem]{Conjecture}
\theoremstyle{remark} \newtheorem{remark}[theorem]{Remark}

\begin{center}
\vskip 1cm{\LARGE\bf An Identity for Generalized \\
\vskip .1in
Euler Polynomials\\
}
\vskip 1cm
\large
Redha Chellal\\
LATN, Faculty of Mathematics\\
USTHB\\
Algiers\\
Algeria\\
rchellal@usthb.dz\\
chellalredha4@gmail.com

\end{center}

\vskip .2 in

\begin{abstract}
In this paper, we introduce a novel identity for generalized Euler polynomials, leading to further generalizations for several relations involving classical Euler numbers, Euler polynomials, Genocchi polynomials, and Genocchi numbers.
\end{abstract}

\section{Introduction}\label{sec:intro}

Let $\mathbb{N}$ and $\mathbb{C}$ denote, respectively, the set of positive integers and the set of complex numbers. Roman \cite{rom} defined generalized
Euler polynomials $E_{n}^{(\alpha)}(x)$ as
follows: for all $n \in \mathbb{N}$ and $\alpha \in \mathbb{C}$,
\begin{equation}
\sum_{n=0}^{\infty }E_{n}^{(\alpha)}(x)\frac{t^{n}
}{n!}=\bigg(\frac{2}{e^{t}+1}\bigg)^{\alpha}e^{tx}.
\end{equation}

We denote by $D$ the classical derivation operator and by $\Psi _{\alpha}$ the automorphism of the vector space
$\mathbb{C}[x]$ defined for $\alpha \in \mathbb{C}$ by
\begin{equation}
\Psi _{\alpha}(x^{n})=E_{n}^{(\alpha)}(x), \text{\ }n\in \mathbb{N}.
\end{equation}
In this paper, we are mainly interested in evaluating the sum $T_{n,\ell
,r}^{(\alpha)}(x,y,z)$ defined by
\begin{align*}
T_{n,\ell ,r}^{(\alpha)}(x,y,z) &=\sum_{k=0}^{n+r}
\binom{n+r}{k}\binom{\ell +k+r}{r}x^{n+r-k}E_{\ell +k}^{(\alpha
)}(y) \\
&+(-1)^{\ell +n+r+1+s}\sum_{k=0}^{\ell +r}\binom{\ell +r}{k}
\binom{n+k+r}{r}x^{\ell +r-k}E_{n+k}^{(\alpha)}(z).
\end{align*}
For all complex numbers $\alpha$, $\lambda$, and for all $l$, $n$, $r$, and $s$ are non-negative
integers. One establishes a result analogous to \cite[Thm.\ 3, Eq.\ 42, p.\ 8]{che}.

\section{Some properties of the generalized Euler polynomials and lemmas}\label{sec:some}

Let us consider the three following operators defined over any endomorphism of the vector space $\mathbb{C}[x]$.
The classical derivation operator $D$, the identity operator $I$, and the operator $\Lambda$ are, respectively, defined by
\begin{equation}
D(x^{n})=nx^{n-1},\text{ for } n\geq 1\text{\ } and\text{\ } 0 \text{ otherwise }.
\end{equation}
\begin{equation}
I(x^{n})=x^{n}\text{ and } \Lambda(x^{n})=(x+1)^{n}+x^{n},\text{\ }n\in \mathbb{N}.
\end{equation}
Generalized Euler polynomials form a sequence of Appell polynomials \cite{appl},
satisfy the following well-know properties \cite{rom}:
\begin{align}
E_{0}^{(\alpha)}(x)&=1,\\
D(E_{n}^{(\alpha)}(x))&=nE_{n-1}^{(\alpha)}(x),\text{\ }n\geq 1, \\
E_{n}^{(\alpha)}(x+y)&=\sum_{k=0}^{n}\binom{n}{k}
y^{n-k}E_{n}^{(\alpha)}(x), \\
\Lambda(E_{n}^{( \alpha)}(x))&=2E_{n}^{(\alpha -1) }(x), \\
E_{n}^{(\alpha)}(\alpha -x)&=(-1)^{n}E_{n}^{(\alpha)}(x).
\end{align}
\begin{lemma}
For every non-negative $n$ and for all complex numbers $\alpha $ and $\gamma
$, one has
\begin{equation}
\Psi _{\alpha }((x+\gamma)^{n}) =E_{n}^{(
\alpha)}( x+\gamma).
\end{equation}
\end{lemma}

\begin{lemma}
\label{lem2}For every $\alpha \in \mathbb{C}$, one has
\begin{align}
D\circ \Psi _{\alpha} &=\Psi_{\alpha}\circ D,  \label{r7} \\
\Psi _{\alpha}\circ \Lambda &=2\Psi _{\alpha -1}.  \label{r8}
\end{align}
\end{lemma}

\begin{proof}[Proof of Eq.~(\ref{r7})]
As $(E_{n}^{(\alpha)}(x))$ is an Appell polynomial, it follows:
\begin{align*}
(D\circ \Psi_{\alpha})(x^{n})&=D(E_{n}^{(\alpha)}(x)) \\
&=nE_{n-1}^{(\alpha)}(x) \\
&=n\Psi_{\alpha}(x^{n-1}) \\
&=\Psi_{\alpha}(nx^{n-1}) \\
&=(\Psi_{\alpha}\circ D)(x^{n}) .
\end{align*}
\end{proof}

\begin{proof}[Proof of Eq.~(\ref{r8})]
It is easy to prove that
\begin{equation*}
E_{n}^{(\alpha)}(x+1) +E_{n}^{(\alpha)}(x) =2E_{n}^{(\alpha -1)}(x) .
\end{equation*}
We deduce that
\begin{align*}
(\Psi _{\alpha }\circ \Lambda)(x^{n})&=\Psi_{\alpha}((x+1)^{n}+x^{n}) \\
&=E_{n}^{(\alpha)}(x+1)+ E_{n}^{(\alpha)}(x) \\
&=2E_{n}^{(\alpha -1)}(x) \\
&=2\Psi _{\alpha -1}(x^{n}) .
\end{align*}
\end{proof}

\section{Main result}\label{sec:main}

The following theorem provides some simplified expressions of $T_{n,\ell
,r}^{(\alpha)}(x,y,z)$ for $x+y+z-\alpha =-s$
where $s$ is non-negative integer.
\begin{theorem}\label{Th1}
For all complex numbers $\alpha$, $\lambda$, and for all non-negative
integers $\ell$, $n$, $r$, and $s$, one has
\begin{align}
&\sum_{k=0}^{n+r}\lambda ^{n+r-k}\binom{n+r}{k}\binom{\ell +k+r}{r}E_{\ell
+k}^{(\alpha)}(x) \notag \\
&+(-1)^{\ell +n+r+s+1}\sum_{k=0}^{\ell +r}\lambda ^{\ell +r-k}%
\binom{\ell +r}{k}\binom{n+k+r}{r}E_{n+k}^{(\alpha)}(\alpha -s-\lambda -x) \notag \\
&=2\Psi_{\alpha -1}\bigr(\frac{D^{r}}{r!}\sum_{k=0}^{s-1}(-1)
^{k}(x+k)^{\ell +r}(x+\lambda +k)^{n+r}\bigr).  \label{le1}
\end{align}
\end{theorem}
It is observed that Eq.~(\ref{le1}) can be formulated as follows:
\begin{equation}
T_{n,\ell ,r}^{(\alpha)}(\lambda ,x,\alpha -s-\lambda
-x)=2\sum_{k=0}^{s-1}\sum_{j=0}^{n+r}(-1)^{k}\lambda
^{n+r-j}\binom{n+r}{j}\binom{\ell +r+j}{r}E_{\ell +j}^{(\alpha-1)}(x+k).
\end{equation}

\begin{proof}
Let us consider the polynomial $Q(x)$ defined by
\begin{equation*}
Q(x) =\sum_{k=0}^{s-1}Q_{k}(x)
\end{equation*}
where
\begin{equation*}
Q_{k}(x)=(-1)^{k}\frac{D^{r}}{r!}\bigr((x+k)^{\ell +r}(x+\lambda +k)^{n+r}\bigr).
\end{equation*}
In this context, it is worth noting the equality $Q_{k}(x+1)=-Q_{k+1}(
x)$ and thus
\begin{align*}
Q(x+1)+Q(x) &=\sum_{k=0}^{s-1}Q_{k}(x)-\sum_{k=0}^{s-1}Q_{k+1}(x) \\
&=\sum_{k=0}^{s-1}Q_{k}(x) -\sum_{k=1}^{s}Q_{k}(x) \\
&=Q_{0}(x)-Q_{s}(x).
\end{align*}
One has
\begin{align*}
Q_{0}(x)&=\frac{D^{r}}{r!}\bigr(x^{\ell +r}(x+\lambda)^{n+r}\bigr) \\
&=\frac{D^{r}}{r!}\bigr(\sum_{k=0}^{n+r}\lambda^{n+r-k}\binom{n+r}{k}x^{\ell+k+r}\bigr) \\
&=\sum_{k=0}^{n+r}\lambda ^{n+r-k}\binom{n+r}{k}\binom{\ell +k+r}{r}x^{\ell+k}
\end{align*}
and
\begin{align*}
Q_{s}(x)&=(-1)^{s}\frac{D^{r}}{r!}\bigr(((x+\lambda +s)-\lambda)^{\ell +r}(x+\lambda +s)^{n+r}\bigr) \\
&=(-1)^{s}\frac{D^{r}}{r!}\bigr(\sum_{k=0}^{n+r}(-\lambda)^{\ell +r-k}\binom{\ell +r}{k}(x+\lambda +s)^{n+k+r}\bigr) \\
&=(-1)^{\ell +n+r+s}\sum_{k=0}^{n+r}\lambda ^{\ell +r-k}\binom{\ell +r}{k}\binom{n+k+r}{r}(-1)^{n+k}(x+\lambda +s)^{n+k}
\end{align*}
thus
\begin{align*}
Q(x+1)+Q(x) &=Q_{0}(x)-Q_{s}(x) \\
&=\sum_{k=0}^{n+r}\lambda^{n+r-k}\binom{n+r}{k}\binom{\ell +k+r}{r}x^{\ell+k} \\
&-(-1)^{\ell +n+r+s}\sum_{k=0}^{\ell +r}\lambda ^{\ell +r-k}\binom{\ell +r}{k}\binom{n+k+r}{r}(-1)^{n+k}(x+\lambda+s)^{n+k}.
\end{align*}
So
\begin{align*}
&\Psi_{\alpha}(Q(x+1)+Q(x))\\
&=\sum_{k=0}^{n+r}\lambda^{n+r-k}\binom{n+r}{k}\binom{\ell +k+r}{r}E_{\ell
+k}^{(\alpha)}(x) \\
&-(-1)^{\ell +n+r+s}\sum_{k=0}^{\ell +r}\lambda^{\ell +r-k}
\binom{\ell +r}{k}\binom{n+k+r}{r}(-1)^{n+k}E_{n+k}^{(\alpha)}(x+\lambda +s) \\
&=\sum_{k=0}^{n+r}\lambda^{n+r-k}\binom{n+r}{k}\binom{\ell +k+r}{r}E_{\ell+k}^{(\alpha)}(x) \\
&-(-1)^{\ell +n+r+s}\sum_{k=0}^{\ell +r}\lambda^{\ell +r-k}\binom{\ell +r}{k}\binom{n+k+r}{r}E_{n+k}^{(\alpha)}(\alpha -\lambda -s-x).
\end{align*}
By employing Eq.~(\ref{r8}) in Lemma \ref{lem2}, it can be derived that
\begin{align*}
\Psi _{\alpha }(Q(x+1)+Q(x))&=2\Psi_{\alpha -1}(Q(x)) \\
&=2\Psi _{\alpha -1}(\sum_{k=0}^{s-1}(-1)^{k}\frac{D^{r}
}{r!}((x+k)^{\ell +r}(x+\lambda +k)^{n+r}) \\
&=2\Psi _{\alpha -1}(\frac{D^{r}}{r!}\sum_{k=0}^{s-1}(-1)^{k}(x+k)^{\ell +r}(x+\lambda +k)^{n+r}).
\end{align*}
\end{proof}

\section{Applications}\label{sec:App}

When $s=0$, Theorem \ref{Th1} gives rise to the following corollary:

\begin{corollary}\label{Cor1}
For every complex numbers $\alpha $, one has
\begin{align}
&\sum_{k=0}^{n+r}x^{n+r-k}\binom{n+r}{k}\binom{\ell +k+r}{r}E_{\ell
+k}^{(\alpha)}(y)   \notag \\
&=(-1)^{\ell +n+r}\sum_{k=0}^{\ell +r}x^{\ell +r-k}\binom{\ell
+r}{k}\binom{n+k+r}{r}E_{n+k}^{(\alpha)}(\alpha-x-y) .  \label{r4}
\end{align}
\end{corollary}
Then, in particular, for $\alpha =1$, $r=0$, and $x=1$, one gets
\begin{equation}
(-1)^{n}\sum_{k=0}^{n+r}\binom{n}{k}E_{\ell +k}(y)
=(-1)^{\ell}\sum_{k=0}^{\ell +r}\binom{\ell}{k}E_{n+k}(-y) ,  \label{r2}
\end{equation}
for $\alpha =1$, $r=1$, and $x=1$, one obtains
\begin{equation}
(-1)^{n}\sum_{k=0}^{n+1}\binom{n+1}{k}(\ell+k+1)E_{\ell +k}(y)+(-1)^{\ell}\sum_{k=0}^{\ell +1}
\binom{\ell +1}{k}(n+k+1)E_{n+k}(-y)=0.  \label{r24}
\end{equation}
Eq.~(\ref{r24}) can be formulated as follows:
\begin{align}
&(-1)^{n}\sum_{k=0}^{n}\binom{n+1}{k}( \ell +k+1)E_{\ell +k}(y)  \notag \\
&+(-1)^{\ell}\sum_{k=0}^{\ell }\binom{\ell +1}{k}(
n+k+1)E_{n+k}(-y)  \notag \\
&=( -1) ^{n+1}2( n+\ell +1+2)(E_{n+\ell
+1}(y)-y^{n+\ell +1}) .  \label{r3}
\end{align}
Wu, Sun, and Pan \cite{wu} found, in 2004, Equations (\ref{r2}) and (\ref{r3}).
For $r=0$ and $\alpha =1$,  Eq.~(\ref{r4}) becomes
\begin{equation}
(-1)^{n}\sum_{k=0}^{n}x^{n-k}\binom{n}{k}E_{\ell +k}(
y) =( -1)^{\ell}\sum_{k=0}^{\ell }x^{\ell -k}\binom{\ell
}{k}E_{n+k}( 1-x-y) .  \label{r5}
\end{equation}
In 2003, Sun \cite[Thm.\ 1.2, Eq.\ (iii)]{sun} obtained Eq.~(\ref{r5}).

\begin{corollary}\label{cor0}Let $k,q,m,n\in \mathbb{N}$, one has
\begin{align}
&(-1)^{m}\sum_{i=0}^{m+q}\alpha^{n+q-i}\binom{m+q}{i}\binom{
n+q+i}{k}E_{n+q+i-k}^{(\alpha)}(x)  \notag \\
&+(-1)^{n+k+1}\sum_{j=0}^{n +q}\alpha ^{n+q-j}\binom{n+q}{j}
\binom{m+q+j}{j}E_{m+q+j-k}^{(\alpha)}(-x) =0. \label{r6}
\end{align}
\end{corollary}

\begin{proof}
For $s=0$ and $\lambda =\alpha$, Theorem \ref{Th1} leads us to
\begin{align}
&\sum_{k=0}^{n+r}\alpha ^{n+r-k}\binom{n+r}{k}\binom{\ell +k+r}{r}E_{\ell
+k}^{(\alpha)}(x)  \notag \\
&+(-1)^{\ell+n+r+1}\sum_{k=0}^{\ell +r}\alpha ^{\ell +r-k}\binom{\ell +r}{k}\binom{n+k+r
}{r}E_{n+k}^{(\alpha)}(-x) =0 \notag.
\end{align}
Changing $n$ to $m+q-k$, $\ell$ to $n+q-k$, and $r$ to $k$,
one obtains Eq.~(\ref{r6}). This completes the proof of corollary (\ref{cor0}).
\end{proof}

For $\alpha =1$ and $k$ odd, Corollary \ref{cor0} allows us to deduce the following equation, S. Hu and M-S Kim \cite[Thm.\ 1.1, p.\ 3]{hu}
proved this equation.
\begin{align}
&(-1)^{m}\sum_{i=0}^{m+q}\binom{m+q}{i}\binom{n+q+i}{k}
E_{n+q+i-k}(x) \notag\\
&+(-1)^{n}\sum_{j=0}^{\ell +r}\binom{n+q}{j}\binom{m+q+j}{k}
E_{m+q+j-k}(-x) =0.
\end{align}

\begin{corollary}\label{Cor2}
\begin{align}
&(-1)^{\ell }\sum_{k=0}^{n+r}\binom{n+r}{k}\binom{\ell +k+r}{r}
2^{n+r-1-k}E_{\ell +k} \notag \\
&+(-1)^{n+r}\sum_{k=0}^{\ell +r}\binom{\ell +r}{k}\binom{n+k+r
}{r}2^{\ell +r-1-k}E_{n+k} \notag \\
&=\sum_{j=0}^{r}(-1)^{j}\binom{n+r}{j}\binom{\ell +r}{r-j}.
\end{align}
\end{corollary}
One obtains \cite[Eq.\ (4.3)(ii), p.\ 210]{ago1}.

\begin{corollary} \label{Cor3}
For $r$ even, one has
\begin{equation}
\sum_{k=0}^{n+r}\binom{n+r}{k}\binom{n+r+k}{r}2^{n+r-k}E_{n+k}=(-1)^{n+\frac{
r}{2}}\binom{n+r}{\frac{r}{2}}.
\end{equation}
\end{corollary}
One finds \cite[Eq.\ (4.4)(ii), p.\ 210]{ago1}.

Exploiting the relation $G_{n}(x)=nE_{n-1}(x)$, for $n\geq 1$ where $G_{n}(x)$ is a classical
Genocchi polynomial defined by
\begin{equation}
\sum_{n=0}^{\infty}G_{n}(x)\frac{t^{n}}{n!}=\frac{2t}{e^{t}+1}e^{tx}.
\end{equation}
The classical Genocchi numbers form the sequence of numbers $(G_{n})_{n\geq 1}$, verify
\begin{equation*}
G_{n}=G_{n}(0),\text{\ }n\in \mathbb{N}.
\end{equation*}
Therefore, they are integers and appear in the OEIS ({\it On-Line Encyclopedia of Integer Sequences})
\cite{slo} as \seqnum{A036968}. Such that $E_{n}=2^{n}E_{n}(\frac{1}{2})$ with $(E_{n})_{n\in \mathbb{N}}$ is the sequence of
Euler numbers appear in the OEIS as \seqnum{A000364}. So, the Corollary \ref{Cor2}leads us to the following corollary:

\begin{corollary}\label{Cor4}
\begin{align}
&(-1)^{\ell }\sum_{k=0}^{n+r}\binom{n+r}{k}\binom{\ell +k+r}{r-1}
2^{n+r-1-k}G_{\ell +k+1} \notag \\
&+(-1) ^{n+r}\sum_{k=0}^{\ell +r}\binom{\ell +r}{k}\binom{n+k+r}{r-1}2^{\ell +r-1-k}G_{n+k+1} \notag \\
&=r\sum_{j=0}^{r}(-1)^{j}\binom{n+r}{j}\binom{\ell +r}{r-j}.
\end{align}
\end{corollary}
By applying Corollary \ref{Cor3}, one obtains the following corollary:

\begin{corollary}\label{Cor5}
\begin{equation}
\sum_{k=0}^{n+r}\binom{n+r}{k}\binom{n+r+k}{r-1}2^{n+r-k}G_{n+k+1}=r(-1)^{n+
\frac{r}{2}}\binom{n+r}{\frac{r}{2}}.
\end{equation}
\end{corollary}

Alternatively, substituting $q$ with $0$ in Corollary \ref{cor0} yields the following corollary:
\begin{corollary}\label{Cor6}
\begin{equation}
\sum_{k=0}^{n}\binom{n}{k}G_{\ell +k}(x)+(-1)
^{\ell +n}\sum_{k=0}^{\ell }\binom{\ell}{k}G_{n+k}(-x) =0.
\end{equation}
\end{corollary}

\section{Conclusion}
In this paper, we present a crucial theorem concerning generalized Euler polynomials. Our method, distinguished by its innovative use of composition operators, It enabled the derivation of a comprehensive identity, encapsulating numerous well-known results as special cases. This approach holds promise for establishing analogous theorems for other specific sequences of Appell polynomials exhibiting similar properties.

\bigskip
\hrule
\bigskip

\noindent 2010 {\it Mathematics Subject Classification}: Primary 11B68;
Secondary 05A10, 11B65.


\begin{thebibliography}{99}

\bibitem{ago1} T. Agoh, Recurrences for Bernoulli and Euler polynomials and
numbers. \textit{Expo. Math}. \textbf{18} (2000), 197--214.

\bibitem{appl} P. Appell, Sur une classe de polyn\^{o}mes,
\textit{Ann. Sci. Ec. Norm. Sup\'{e}r}. \textbf{9} (2) (1880) 119--144.

\bibitem{che} R. Chellal, F. Bencherif, and M. Mehbali, An identity for
generalized Bernoulli polynomials, \textit{J. Integer Sequences} \textbf{23} (2020),\href{https://cs.uwaterloo.ca/journals/JIS/VOL23/Chellal/chellal7.html}{Article 20.11.2}.

\bibitem{hu} Hu. Su and Min-Soo Kim, Identities for the Euler polynomials, $
p$-adic integrals and Witt's formula.  Preprint, 2021. Available at \url{https://arxiv.org/abs/2106.01119}.

\bibitem{rom} S. Roman, \textit{The Umbral Calculus}, Academic Press, New York, NY,
USA, 1984.

\bibitem{slo} N. J. A. Sloane et al., {\it The On-Line Encyclopedia of
Integer Sequences}, published electronically at
  \url{https://oeis.org}, 2020.

\bibitem{sun} Z. W. Sun, Combinatorial identities in dual sequences,
\textit{European J. Combin.} \textbf{24} (2003) 709--718.

\bibitem{wu} K. -J. Wu, Z.-W. Sun, and H. Pan, Some identities for
Bernoulli and Euler polynomials, \textit{Fibonacci Quart}. \textbf{42}
(2004), 295--299.
\end{thebibliography}
\end{document}